\documentclass[11pt]{article}

\usepackage[mathscr]{eucal}
\usepackage{color}
\usepackage{amsmath}
\usepackage{amssymb}
\usepackage{amsfonts}
\usepackage{amscd}
\usepackage{amsthm}
\usepackage{latexsym}
\usepackage{graphicx}
\usepackage{subfig}

\def\be{\begin{equation}}
\def\ee{\end{equation}}
\def\bse{\begin{subequations}}
\def\ese{\end{subequations}}

\usepackage{latexsym}

\let\er\eqref

\let\be\beta

\newcommand{\R}{{\mathbb R}}

\newtheorem{theorem}{Theorem}
\newtheorem{lemma}[theorem]{Lemma}
\newtheorem{proposition}[theorem]{Proposition}

\newtheorem{definition}[theorem]{Definition}

\def\bse{\begin{subequations}}
\def\ese{\end{subequations}}

\usepackage{geometry}

\title{The aggregation-diffusion equation with the intermediate exponent}

\author{ Shen Bian\footnote{Corresponding author. Department of Mathematical Sciences, Beijing University of Chemical Technology, Beijing, 100029, P.R. CHINA. Email address: bianshen66@163.com.}
  \and Jiale Bu\footnote{Department of Mathematical Sciences, Beijing University of Chemical Technology, Beijing, 100029, P.R. CHINA. Email address: 2022201105@buct.edu.cn.} }
\date{}

\begin{document}
\let\cleardoublepage\clearpage

\maketitle

\begin{abstract}
We consider a Keller-Segel model with non-linear porous medium type diffusion and nonlocal attractive power law interaction, focusing on potentials that are less singular than Newtonian interaction. Here, the nonlinear diffusion is chosen to be $\frac{2d}{d+2s}<m<2-\frac{2s}{d}$ in which case the steady states are compactly supported. We analyse under which conditions on the initial data the regime that attractive forces are stronger than diffusion occurs and classify the global existence and finite time blow-up of solutions. It is shown that there is a threshold value which is characterized by the optimal constant of a variant of Hardy-Littlewood-Sobolev inequality such that the solution will exist globally if the initial data is below the threshold, while the solution blows up in finite time when the initial data is above the threshold.
\end{abstract}

\section{Introduction}
The present paper concerns the following aggregation-diffusion equation
\begin{align}\label{fracks}
\left\{
      \begin{array}{ll}
      u_t = \Delta u^m-\nabla \cdot \left(u \nabla c\right), ~~& x\in \R^{d},~t> 0,\\
     u(x,0) =u_{0}(x) \ge 0, ~~& x\in \R^d.
      \end{array}\right.
\end{align}
Here $d \ge 3$ and the diffusion exponent $m>1$. $u$ represents the density of cells, the chemoattractant $c$ is governed by a fractional diffusion process and can be expressed by the convolution of $u(x,t)$ with Riesz potential, that's
\begin{align}\label{CC}
c(x,t) =(-\Delta)^{-s}u=\mathcal{K}  \ast u = c_{d,s}\int_{\mathbb R^d} \frac{u(y,t)}{|x-y|^{d-2s}} dy,\quad c_{d,s}=\frac{\Gamma(d/2-s)}{\pi^{d/2}4^s \Gamma(s)},
\end{align}
where the potential is less singular than Newtonian interaction, i.e.
\begin{align}\label{s1}
2<2s<d.
\end{align}
The parabolic equation on $u$ gives nonnegative solutions $u(x,t)\ge 0, c(x,t)\ge 0$. Another property we will use is the conservation of the total number of cells
\begin{align}
M:=\int_{\R^d} u_0(x)dx=\int_{\R^d}u(x,t) dx.
\end{align}
Initial data will be assumed throughout this paper
\begin{align}\label{12}
u_0 \in L_+^1 \cap L^\infty(\R^d), ~~ \int_{\R^d} |x|^2 u_0(x) dx<\infty, ~~ \nabla u_0^m \in L^2(\R^d).
\end{align}
The nonlinear diffusion $m>1$ models the local repulsion of cells while the attractive nonlocal term models cell movement toward chemotactic sources or attractive interaction between cells due to cell adhesion \cite{BL98,FW03,GC08,KS70,MT15,PB15}. The main characteristic of equation \er{fracks} is the competition between the diffusion and the nonlocal aggregation which may lead to finite-time blow-up or global existence. This is well presented by the free energy
\begin{align}\label{Fu}
F(u)=\frac{1}{m-1} \int_{\R^d} u^m dx-\frac{c_{d,s}}{2} \iint_{\R^d \times \R^d} \frac{u(x,t)u(y,t)}{|x-y|^{d-2s}} dxdy.
\end{align}
As a matter of fact, \er{fracks} can be recast as
\begin{align}\label{chemical}
u_t=\nabla \cdot \left( u\nabla \mu  \right),
\end{align}
where $\mu$ is the chemical potential
\begin{align}
\mu=\frac{m}{m-1}u^{m-1}-c.
\end{align}
Multiplying \er{chemical} by $\mu$ and integrating it in space yield
\begin{align}
\frac{d F(u)}{dt}+\int_{\R^d} u|\nabla \mu|^2 dx=0.
\end{align}
This implies that $F[u(\cdot,t)]$ is non-increasing with respect to time.

Let us mention that equation \er{fracks} possesses a scaling invariance which leaves the $L^p$ norm invariant where
\begin{align}
p:=\frac{d(2-m)}{2s}
\end{align}
and produces a balance between the diffusion and the aggregation terms. Indeed, if $u(x,t)$ is a solution to \er{fracks}, then $u_\lambda(x,t)=\lambda^{\frac{2s}{2-m}} u \left( \lambda x,\lambda^{2+\frac{2s(m-1)}{2-m}} t \right)$ is also a solution to \er{fracks} and this scaling preserves the $L^p$ norm $\|u_\lambda\|_{L^p(\R^d)}=\|u\|_{L^p(\R^d)}.$ The exponent $m=2-2s/d$ produces a balance in the mass invariant scaling $u_\lambda=\lambda^d u\left( \lambda x,\lambda^{d+2-2s} t \right)$ of the diffusion term
$\lambda^{dm+2} \Delta u_\lambda^m$ and the aggregation term $\lambda^{2d+2-2s} \nabla \cdot \left( u_\lambda \nabla  \mathcal{K} \ast u_\lambda \right)$. Under the mass invariant scaling, we have that for the subcritical case $m>2-2s/d$, the aggregation dominates the diffusion for low density and prevents spreading. While for high density, the diffusion dominates the aggregation thus blow-up is precluded. On the contrary, for the supercritical case $1 < m<2-2s/d,$ the aggregation dominates the diffusion for high density (large $\lambda$) and the density may have finite-time blow-up. While for low density (small $\lambda$), the diffusion dominates the aggregation and the density has infinite-time spreading. These behaviors also appear in many other physical systems such as thin film, Hele-Shaw, stellar collapse as well as the nonlinear Schr\"{o}dinger equation \cite{BC99,Chandra,Wei83,Tom04}.

Our main goal in this paper is to classify the global existence and finite time blow-up of solutions by the initial data for the supercritical case $1<m<2-2s/d$. The main tool for the analysis of \er{fracks} is a variant of the Hardy-Littlewood-Sobolev inequality, see Section \ref{MHLS}: for all $f \in L_+^1 \cap L^m (\R^d),$ there exists an optimal constant $C_*$ such that
\begin{align}\label{VHLS}
\iint_{\R^d \times \R^d} \frac{f(x)f(y)}{|x-y|^{d-2s}} dxdy  \le C_* \|f\|_{L^1(\R^d)}^{\frac{(d+2s)m-2d}{d(m-1)}} \|f\|_{L^m(\R^d)}^{\frac{m(d-2s)}{d(m-1)}}.
\end{align}
The identification of the equality case in \er{VHLS} is given by a non-negative, radially symmetric and non-increasing function $W \in L_+^1 \cap L^m(\R^d)$, whence
\begin{align}\label{WW}
 \iint_{\R^d \times \R^d} \frac{W(x)W(y)}{|x-y|^{d-2s}} dxdy  = C_* \|W\|_{L^1(\R^d)}^{\frac{(d+2s)m-2d}{d(m-1)}} \|W \|_{L^m(\R^d)}^{\frac{m(d-2s)}{d(m-1)}}.
\end{align}
Moreover, Lemma \ref{existencecstar} will show that $W(x)$ is the radial steady state of \er{fracks}. Actually, the steady state $W$ can be interpreted in the distributional sense \cite{B2023}
\begin{align}\label{steadystates}
\left\{
      \begin{array}{ll}
       \frac{m}{m-1} W^{m-1}-C_s=\overline{C}, ~~\mbox{in} ~~\Omega, \\
        W=0 \mbox{~ in ~} \R^d \setminus \Omega, ~~ W>0 \mbox{ in } \Omega, \\
        C_s = c_{d,s} \int_{\mathbb R^d} \frac{W(y)}{|x-y|^{d-2s}} dy, ~~\mbox{in}~~ \R^d.
      \end{array}\right.
\end{align}
Here $\overline{C}$ is any constant chemical potential and $\Omega=\{x \in \R^d ~\big|~ W(x)>0 \}$ is a connected open set in $\R^d$. We stress that a lot of efforts have been put \cite{CCH17,CCH21,CHVY19,CHM18,CGH20} to investigate the steady states of \er{fracks} in terms of the diffusion exponent $m$.

For $m>2-2s/d,$ this is the diffusion-dominated regime. For $d \ge 1$ and $0<2s<d$, \cite{CHM18} showed that all stationary states of \er{fracks} are radially symmetric non-increasing, compactly-supported and they are global minimizers of the energy functional $F(u).$ Furthermore, uniqueness of steady states is obtained for $m>2-2s/d, d=1, 0<2s<1$ in \cite{CCH21,CHM18} and for $m \ge 2, 0<2s \le d+1$ in \cite{DYY22}. While non-uniqueness of steady states is proved for $2-2s/d<m<2, s \ge 1$ in \cite{CCH21}. In the local case $s=1,$ the Riesz kernel $\frac{1}{|x|^{d-2s}}$ is replaced by Newtonian interaction for $d \ge 3$ and log interaction for $d=2.$ Analogous results for stationary solutions are provided in \cite{CHVY19,kimyao} for $m>2-2/d~(d \ge 3)$ and \cite{CCV15} for $m>1~(d=2)$ respectively.

In the case $m=2-2s/d,$ this is the fair competition regime where the nonlinear diffusion and the nonlocal aggregation balance. The existence of the stationary states can only happen for a critical mass $M_c$ which is characterised by the optimal constant of a variant of the HLS inequality. $M_c$ turns out to be the only value of the mass for which the free energy $F(u)$ has minimizers and these minimizers are shown to be compactly supported, radially symmetric and non-increasing stationary solutions of equation \er{fracks}. See for instance \cite{CCH17} and the references therein.

For $1<m<2-2s/d,$ this is the aggregation-dominated regime. In this work, we will focus on the case $2d/(d+2s)<m<2-2s/d$ where all the stationary states are compactly supported (while they are not compactly supported for $1<m \le 2d/(d+2s)$). Here the invariant exponent $p$ turns into
\begin{align}
1<p=\frac{d(2-m)}{2s}<\frac{2d}{d+2s}<m.
\end{align}
We will make a fundamental use of the invariant space $L^p(\R^d)$ and the modified Hardy-Littlewood-Sobolev inequality \er{VHLS} to classify dynamical behaviors of solutions.

Before showing the main result, we firstly define the weak solution which we will deal with throughout this paper.
\begin{definition}\label{weakdefine}(Weak and free energy solution)
Let $u_0$ be an initial condition satisfying \er{12} and $T \in (0,\infty].$
\begin{enumerate}
\item[\textbf{(i)}]
  A weak solution to \er{fracks} with initial data $u_0$ is a non-negative function $u \in L^\infty \left(0,T;L_+^1 \cap L^\infty(\R^d) \right)$ and for any $0<t<T$,
 \begin{align}
 &\int_{\R^d} \psi u(\cdot,t)dx-\int_{\R^d} \psi u_0(x) dx =\int_0^t
 \int_{\R^d} \Delta \psi  u^m dx ds \nonumber \\
 & - \frac{c_{d,s}(d-2s)}{2} \int_0^t \iint_{\R^d\times \R^d}  \frac{[\nabla
 \psi(x)-\nabla \psi(y)] \cdot (x-y)}{|x-y|^2} \frac{u(x,s)
 u(y,s)}{|x-y|^{d-2s}} dxdy ds \label{weak}
 \end{align}
  for $\forall ~\psi \in C_0^\infty(\R^d)$.
\item[\textbf{(ii)}]
The weak solution $u$ is also a weak free energy solution to \er{fracks} if the following additional regularities
  \begin{align}
   \nabla u^{m-1/2} \in L^2\left(0,T;L^2(\R^d)\right), \label{0201} \\
    u \in L^{3}\left(0,T;L^{\frac{3d}{d-2(1-2s)}}(\R^d) \right) \label{02011}
  \end{align}
are fulfilled. Moreover, $F(u(\cdot,t))$ is a non-increasing function in time and satisfies that for all $t \in (0,T)$,
\begin{align}\label{Fuequality}
F[u(\cdot,t)]+\int_{0}^t \int_{\R^d} \left| \frac{2m}{2m-1}\nabla
u^{m-\frac{1}{2}}-\sqrt{u}\nabla c \right|^2dxds \leq
F(u_0),
\end{align}
  where $c=c_{d,s} \int_{\R^d} \frac{u(y)}{|x-y|^{d-2s}}dy.$
\end{enumerate}
\end{definition}
We emphasize that regularities \er{0201} and \er{02011} are enough to make sense of each term in \er{Fuequality}.
\begin{lemma}\label{uceps}
If $u \in L^{\frac{3d}{d-2(1-2s)}}(\R^d)$ and $c$ is expressed by \er{CC}, then
\begin{align}\label{06223}
\left\| u |\nabla c|^2 \right\|_{L^1(\R^d)} \leq C
\|u\|_{L^q(\R^d)}^3<\infty,~~q=\frac{3d}{d-2(1-2s)}\,.
\end{align}
\end{lemma}
\begin{proof}
To verify \er{06223}, we can utilize H\"{o}lder's inequality to get
\begin{align}\label{qqp1}
\left\| u |\nabla c|^2 \right\|_{L^1(\R^d)} \leq \|u\|_{L^q(\R^d)} \left\||\nabla
c|^2 \right\|_{L^{q'}(\R^d)},  ~~\frac{1}{q}+\frac{1}{q'}=1.
\end{align}
Besides, using the weak Young inequality \cite[formula (9), pp.
107]{lieb202} we continue to get
\begin{align}\label{0623}
\left\||\nabla c|^2\right\|_{L^{q'}(\R^d)} & =\|\nabla c\|_{L^{2q'}(\R^d)}^2=C
\left\|u(x)*\frac{x}{|x|^{d+2-2s}}\right\|_{L^{2q'}(\R^d)}^2 \nonumber \\
& \leq C
\|u\|_{L^{q}(\R^d)}^2 \left\|\frac{x}{|x|^{d+2-2s}}
\right\|_{L_w^{\frac{d}{d+1-2s}}(\R^d)}^2\le C\|u\|_{L^{q}(\R^d)}^2,
\end{align}
where $1+\frac{1}{2q'}=\frac{1}{q}+\frac{d+1-2s}{d}.$ Combining \er{qqp1} with \er{0623} follows \er{06223} with $q=\frac{3d}{d-2(1-2s)}$ and thus completes the proof. $\Box$
\end{proof}

Consequently, from \er{02011} one has that
\begin{align}
\int_{0}^T \left\| \sqrt{u} \nabla c \right\|_{L^2(\R^d)}^2 dt \leq
C\int_{0}^T \left\|u \right\|_{L^{3d/(d-2(1-2s))}(\R^d)}^3 dt<\infty.
\end{align}
So the term $\int_0^t \int_{\R^d} \left|\sqrt{u} \nabla c\right|^2 dx ds$ in \er{Fuequality} makes sense.

Now we are ready to state our main result.
\begin{theorem}\label{main}
Let $d \ge 3$, $\frac{2d}{d+2s}<m<2-\frac{2s}{d}$ and set $a=\frac{(d+2s)m-2d}{2d-2s-dm}.$ Suppose that $W$ satisfies \er{WW}. Under assumption \er{12} and
\begin{align}
\|u_0\|_{L^1(\R^d)}^a F(u_0)<\|W\|_{L^1(\R^d)}^a F(W)
\end{align}
on the initial condition.
\begin{enumerate}
  \item[\textbf{(i)}] If
  \begin{align}
  \|u_0\|_{L^1 (\R^d)}^a \|u_0\|_{L^m (\R^d)}^m< \|W\|_{L^1 (\R^d)}^a \|W\|_{L^m (\R^d)}^m,
  \end{align}
  then there exists a global weak solution to \er{fracks} satisfying that for any $0<t<\infty,$
     \begin{align}
      \|u(\cdot,t)\|_{L^1 \cap L^\infty(\R^d)} \le C\left( \|u_0\|_{L^1 \cap L^\infty(\R^d)} \right).
     \end{align}
     Furthermore, the weak solution is also a weak free energy solution satisfying the energy inequality \er{Fuequality}.
  \item[\textbf{(ii)}] If
  \begin{align}
  \|u_0\|_{L^1 (\R^d)}^a \|u_0\|_{L^m (\R^d)}^m > \|W\|_{L^1 (\R^d)}^a \|W\|_{L^m (\R^d)}^m,
  \end{align}
  then the weak solution of \er{fracks} blows up in finite time $T$ in the sense that
     \begin{align}
       \displaystyle \limsup_{t \to T} \|u(\cdot,t)\|_{L^\infty(\R^d)}=\infty.
     \end{align}
\end{enumerate}
\end{theorem}

One thing is worth noting that the above results also hold true in the local setting $s=1$ for the lower dimensional case $d=2$ where Riesz potential is replaced by the log interaction. The energy critical exponent $\frac{2d}{d+2s}$ and mass critical exponent $2-\frac{2s}{d}$ coincides, i.e. $\frac{2d}{d+2s}=2-\frac{2s}{d}=1.$ In this case, the $L^1$ norm of the steady states is $8\pi$ and the global behaviors of solutions can be classified by the $L^1$ norm of the initial data. Namely, if $\|u_0\|_{L^1(\R^2)} < 8\pi,$ the solution of \er{fracks} exists globally in time, while the solution blows up in finite time for $\|u_0\|_{L^1(\R^2)} > 8\pi$ \cite{BCM08,bdp06}. Let us point out that in the nonlocal case $s>1,$ it is not straightforward to extend the above results to $d=2.$ In fact, there is a gap between $\frac{2d}{d+2s}$ and $2-\frac{2s}{d}$, i.e. $2-\frac{2s}{d}<\frac{2d}{d+2s}$, and the two dimensional degenerate problem is essentially new aspect.

The results are organised as follows. This work is entirely devoted to establish an exact criteria on the initial data for the global existence and finite time blow-up of solutions to \er{fracks}. With this aim we firstly construct the existence criterion which shows a key maximal existence time for the free energy solution of \er{fracks} in Section \ref{Existence}. Section \ref{MHLS} mainly explores the variational structure of the modified HLS inequality. Subsequently, Section \ref{Proofmain} makes use of the extremum function $W$ of the VHLS inequality to prove the main theorem concerning the dichotomy, the global existence for $\|u_0\|_{L^1 (\R^d)}^a \|u_0\|_{L^m (\R^d)}^m< \|W\|_{L^1 (\R^d)}^a \|W\|_{L^m (\R^d)}^m$ and finite time blow-up for $\|u_0\|_{L^1 (\R^d)}^a \|u_0\|_{L^m (\R^d)}^m >\|W\|_{L^1 (\R^d)}^a \|W\|_{L^m (\R^d)}^m$. Finally, Section \ref{conclu} concludes the main work of this paper.


\section{Existence criterion} \label{Existence}

Following \cite{BL13,BL14,suku06}, we consider the regularized problem
\begin{align}\label{kseps}
\left\{
  \begin{array}{ll}
    \partial_t u_\varepsilon=\Delta u^m_\varepsilon+\varepsilon \Delta u_\varepsilon-\nabla \cdot \left( u_\varepsilon \nabla c_\varepsilon \right),~~x \in \R^d,~t>0, \\
 u_\varepsilon(x,0)=u_{0\varepsilon} \ge 0,
  \end{array}
\right.
\end{align}
where $c_\varepsilon$ is defined as
\begin{align}
c_\varepsilon=R_\varepsilon \ast u_\varepsilon
\end{align}
with the regularized Riesz potential
\begin{align}
R_\varepsilon(x)=c_{d,s} \frac{1}{\left( |x|^2+\varepsilon^2  \right)^{\frac{d-2s}{2}}}.
\end{align}
Here $u_{0\varepsilon}$ is a sequence of approximation for $u_0$ and can be constructed to satisfy that there exists $\varepsilon_0>0$ such that for any $0<\varepsilon<\varepsilon_0,$
\begin{align}
\left\{
  \begin{array}{ll}
 u_{0\varepsilon} \ge 0, ~~ \|u_\varepsilon(x,0)\|_{L^1(\R^d)}=\|u_0\|_{L^1(\R^d)}, \\
 \int_{\R^d}|x|^2 u_{0\varepsilon} dx \to \int_{\R^d}|x|^2 u_0(x)dx,~~\mbox{as}~~\varepsilon \to 0, \\
 u_{0\varepsilon}(x) \to u_0(x)~~\mbox{in}~~L^q(\R^d),~~\mbox{for}~~1 \le q <\infty,~~\mbox{as}~~\varepsilon \to 0.
  \end{array}
\right.
\end{align}
The regularized problem has global in time smooth solutions for any $\varepsilon>0$. This approximation has been proved to be convergent. More precisely, following the arguments in \cite[Theorem 4.2]{BL14}, \cite[Lemma 4.8]{CHVY19} and \cite[Section 4]{suku06} we assert that if
\begin{align}\label{Linfinity}
\|u_\varepsilon(\cdot,t)\|_{L^\infty(\R^d)}<C_0,
\end{align}
where $C_0$ is independent of $\varepsilon,$ then there exists a subsequence $\varepsilon_n \to 0$ such that
\begin{align}
  \begin{array}{ll}
    u_{\varepsilon_n} \rightarrow u~~\mbox{in}~~L^r(0,T;L^r(\R^d)),~~1 \le r<\infty \label{conver1}
  \end{array}
\end{align}
and $u$ is a weak solution to \er{fracks} on $[0,T).$

According to the above analysis, a weak solution to \er{fracks} on $[0,T)$ exists when \er{Linfinity} is fulfilled. So we shall focus on establishing the availability of the $L^\infty$-bound. As we will see in the following theorem where the local in time existence and blow-up criteria are constructed, such a bound follows from the $L^r$ norm for $r>p=\frac{d(2-m)}{2s}$ which additionally provides a characterisation of the maximal existence time.

Before going to the proof, we firstly recall the HLS inequality \cite{lieb202}:
\begin{lemma}(HLS inequality)
Let $q,q_1>1, d \ge 3$ and $0<\beta<d$ with $\frac{1}{q}+\frac{1}{q_1}+\frac{\beta}{d}=2.$ Assume $f \in L^q(\R^d)$ and $h \in L^{q_1}(\R^d)$, then there exists a sharp constant $C(d,\beta,q)$ independent of $f$ and $h$ such that
\begin{align}\label{fh}
\left| \iint_{\R^d \times \R^d} \frac{f(x)h(y)}{|x-y|^{\beta}} dxdy   \right| \le C(d,\beta,q) \|f\|_{L^q(\R^d)} \|h\|_{L^{q_1}(\R^d)}.
\end{align}
If $q=q_1=\frac{2d}{2d-\beta},$ then
\begin{align}\label{CHLS}
C(d,\beta,q)=C(d,\beta)=\pi^{\beta/2} \frac{\Gamma\left( d/2-\beta/2 \right)}{\Gamma\left( d-\beta/2 \right)} \left( \frac{\Gamma(d/2)}{\Gamma(d)} \right)^{-1+\beta/d}.
\end{align}
\end{lemma}
Now we are able to show the local existence results.
\begin{theorem}(Local in time existence and blow-up criteria)\label{ueps}
Under assumption \er{12} on the initial condition, there are a maximal existence time $T_w \in (0,\infty]$ and a weak solution $u$ to \er{fracks} on $[0,T_w)$. If $T_w<\infty$, then
\begin{align}\label{criterion}
\|u(\cdot,t)\|_{L^\infty(\R^d)} \to \infty~~\mbox{as}~~t \to T_w.
\end{align}
\end{theorem}
\begin{proof}
The main task in proving the local existence is to get the following a priori estimates.

{\it\textbf{Step 1}} ($L^r$-estimates, $r \in (p,\infty)$) \quad It's obtained by multiplying equation \er{kseps} with $ru_\varepsilon^{r-1}$ that
\begin{align}\label{230127}
& \frac{d}{dt} \int_{\R^d} u_\varepsilon^r dx+\frac{4mr(r-1)}{(m+r-1)^2} \int_{\R^d} \left| \nabla u_\varepsilon^{\frac{m+r-1}{2}} \right|^2 dx+\varepsilon \frac{4(r-1)}{r} \int_{\R^d} \left| \nabla u_\varepsilon^{\frac{r}{2}} \right|^2 dx \nonumber \\
=& c_{d,s}(d-2s)(2s-2)(r-1) \iint_{\R^d \times \R^d} \frac{u_\varepsilon^r(x) u_\varepsilon(y)}{\left( |x-y|^2+\varepsilon^2 \right)^{\frac{d+2-2s}{2}}} dxdy.
\end{align}
Because of $2s>2,$ we handle the right hand side of \er{230127} by applying the HLS inequality \er{fh} with $f=u_\varepsilon^r(x)$ and $h=u_\varepsilon(y)$ such that
\begin{align}\label{cdsr}
\iint_{\R^d \times \R^d} \frac{u_\varepsilon^r(x) u_\varepsilon(y)}{\left( |x-y|^2+\varepsilon^2 \right)^{\frac{d+2-2s}{2}}} dxdy & \le \iint_{\R^d \times \R^d} \frac{u_\varepsilon^r(x) u_\varepsilon(y)}{|x-y|^{d+2-2s}} dxdy \nonumber \\
& \le C_{dsr} \|u_\varepsilon^r\|_{L^{q}(\R^d)} \|u_\varepsilon\|_{L^{q_1}(\R^d)} = C_{dsr} \left\|u_\varepsilon \right\|_{L^{\frac{r+1}{1+(2s-2)/d}}(\R^d)}^{r+1},
\end{align}
where $\frac{1}{q}+\frac{1}{q_1}+\frac{d+2-2s}{d}=2$ and we have used $q_1=rq=\frac{r+1}{1+(2s-2)/d}$, $C_{dsr}$ is a bounded constant depending on $d,s,r$. Furthermore, for
$$
1 \le q_0<r_0<\frac{2d}{d-2},
$$
we have the following GNS inequality \cite{bp07}: there exists a positive constant $C$ such that
\begin{align}
\|v\|_{L^{r_0}(\R^d)} \le C \|\nabla v\|_{L^{2}(\R^d)}^\theta \|v\|_{L^{q_0}(\R^d)}^{1-\theta},\quad \theta=\frac{\frac{1}{q_0}-\frac{1}{r_0}}{\frac{1}{q_0}-\frac{d-2}{2d} },
\end{align}
which we apply with $v=u_\varepsilon^{\frac{m+r-1}{2}}$ to discover
\begin{align}\label{230129}
\left\|u_\varepsilon \right\|_{L^{\frac{r+1}{1+(2s-2)/d}}(\R^d)}^{r+1} \le C \left\|\nabla u_\varepsilon^{\frac{m+r-1}{2}} \right\|_{L^2(\R^d)}^{a \theta} \|u_\varepsilon\|_{L^{q_0 (m+r-1)/2}(\R^d)}^{a(1-\theta)\frac{m+r-1}{2}},
\end{align}
where the parameters satisfy
\begin{align*}
a \frac{m+r-1}{2}=r+1, \quad r_0 \frac{m+r-1}{2}=\frac{r+1}{1+(2s-2)/d}.
\end{align*}
Here we pick
\begin{align*}
q_0 \frac{m+r-1}{2}>p=\frac{d(2-m)}{2s}
\end{align*}
and simple computations show that
\begin{align*}
a \theta=2 \left( 1+ \frac{ \frac{2-m}{(m+r-1)q_0}-\frac{2s}{2d}}{ \frac{1}{q_0}-\frac{d-2}{2d} } \right)<2
\end{align*}
in case of $1 < m < 2-2s/d$. So using Young's inequality we can infer from \er{230129} that
\begin{align}\label{230610}
\left\|u_\varepsilon \right\|_{L^{\frac{r+1}{1+(2s-2)/d}}(\R^d)}^{r+1} \le \frac{2mr}{c_{d,s}(d-2s) C_{dsr} (2s-2) (m+r-1)^2} \left\|\nabla u_\varepsilon^{\frac{m+r-1}{2}} \right\|_{L^2(\R^d)}^2 +C \|u_\varepsilon\|_{L^{q_0 (m+r-1)/2}(\R^d)}^{(r+1)(1-\theta)/(1-a\theta/2)},
\end{align}
where the constant $C_{dsr}$ in \er{cdsr} is used and
\begin{align}
\frac{(r+1)(1-\theta)}{1-a\theta/2}=\frac{2s-(2-m)(2s-2)+2sr-d(2-m)}{2s-\frac{2d(2-m)}{q_0 (m+r-1)}}>q_0 \frac{m+r-1}{2}
\end{align}
in case of
$$\frac{d(2-m)}{2s}<q_0 \frac{m+r-1}{2} \le r.$$
Particularly, taking
$$
q_0 \frac{m+r-1}{2}=r
$$
\er{230610} also reads
\begin{align}\label{0131}
c_{d,s}(d-2s) (2s-2)C_{dsr}(r-1) \|u_\varepsilon\|_{L^{\frac{r+1}{1+(2s-2)/d}}(\R^d)}^{r+1} \le & \frac{2mr(r-1)}{(m+r-1)^2} \left\|\nabla u_\varepsilon^{\frac{m+r-1}{2}} \right\|_{L^2(\R^d)}^2 \nonumber \\
& + C \left( \|u_\varepsilon\|_{L^{r}(\R^d)}^{r} \right)^{1+\frac{\frac{2s}{d}-(2-m)\frac{2s-2}{d}}{\frac{2s}{d}r-(2-m)}}
\end{align}
for any $r>\frac{d(2-m)}{2s}$. Therefore, substituting \er{cdsr} and \er{0131} into \er{230127} we thus end up with
\begin{align}\label{230131}
& \frac{d}{dt} \|u_\varepsilon \|_{L^r(\R^d)}^r +\frac{2mr(r-1)}{(m+r-1)^2} \int_{\R^d} \left| \nabla u_\varepsilon^{\frac{m+r-1}{2}} \right|^2 dx+\varepsilon \frac{4(r-1)}{r} \int_{\R^d} \left| \nabla u_\varepsilon^{r/2} \right|^2 dx \nonumber \\
\le & C \left( \|u_\varepsilon\|_{L^{r}(\R^d)}^{r} \right)^{1+\frac{\frac{2s}{d}-(2-m)\frac{2s-2}{d}}{\frac{2s}{d}r-(2-m)}}.
\end{align}
Hence the local in time $L^r$-estimates is obtained
\begin{align}\label{Lr}
\|u_\varepsilon (\cdot,t)\|_{L^r(\R^d)}^r & \le \frac{1}{\left( \|u_\varepsilon(0) \|_{L^r(\R^d)}^{-r \delta}-\delta C t \right)^{\frac{1}{\delta}}}\nonumber \\
&= \left(\frac{C(\delta)}{T_0-t}\right)^{\frac{1}{\delta}},\qquad T_0=\frac{\|u_\varepsilon(0) \|_{L^r(\R^d)}^{-r \delta}}{\delta C},
\end{align}
where $\delta=\frac{1-\frac{(2-m)(2s-2)}{2s}}{r-\frac{d(2-m)}{2s}}>0$ since $r>\frac{d(2-m)}{2s}$ and $1 < m<2-\frac{2s}{d}.$ Particularly, because $m>\frac{2d}{d+2s}$ implies $m>p=\frac{d(2-m)}{2s}$, the local existence in $L^m$ norm also holds true by letting $r=m$ in \er{Lr}. Furthermore, coming back to \er{230131}, we deduce that
\begin{align}\label{nablaLr}
\left\|\nabla u_\varepsilon^{\frac{m+r-1}{2}} \right\|_{L^2(0,T_0;L^2(\R^d))} \le C\left( \|u_\varepsilon(0)\|_{L^r(\R^d)} \right).
\end{align}

{\it\textbf{Step 2}} ($L^\infty$-estimates)\quad As a direct result of Step 1 with Moser iterative method, we have
\begin{align}\label{Linfinity23}
\displaystyle \sup_{0<t<T_0} \|u_\varepsilon(\cdot,t)\|_{L^\infty(\R^d)} \le C\left( \|u_\varepsilon(0)\|_{L^1(\R^d)},\|u_\varepsilon(0)\|_{L^\infty(\R^d)} \right)
\end{align}
by a word for word translation of the proof for \cite[Theorem 4.2]{BL14}. Armed with regularities \er{nablaLr} and \er{Linfinity23} it directly follows
\begin{align}
\|\nabla u_\varepsilon \|_{L^2(0,T_0;L^2(\R^d))} \le C, \\
\|u_\varepsilon \nabla c_\varepsilon \|_{L^\infty(0,T_0;L^\infty(\R^d))} \le C.
\end{align}
Here $C$ are constants depending only on $\|u_\varepsilon(0)\|_{L^1(\R^d)}$ and $\|u_\varepsilon(0)\|_{L^\infty(\R^d)}$. As a consequence, the a priori bounds in the theorem hold true uniformly in $\varepsilon$ and thus we can pass to the limit $u_\varepsilon \to u$ as $\varepsilon \to 0$ (without relabeling) because of the Lions-Aubin lemma which provides time compactness. In the limit, we obtain the local existence of the weak solution to \er{fracks}. Finally, in the light of the proof for \cite[Theorem 2.4]{BJ09} we characterise the maximal existence time of the weak solution and complete the proof. $\Box$
\end{proof}

Next we may proceed to show that the weak solution is also a free energy solution satisfying $F(u(\cdot,t)) \le F(u_0)$ for $t \in [0,T_w).$

\begin{proposition}(Existence of the free energy solution)\label{energy}
Under assumption \er{12} on the initial data and $u_\varepsilon(\cdot,t) \in L^\infty \left(0,T_w;L_+^1 \cap L^\infty(\R^d) \right)$ on the approximated sequence, the weak solution of \er{fracks} is also a free energy solution satisfying the energy inequality \er{Fuequality}.
\end{proposition}
\begin{proof}
Multiplying $\mu_\varepsilon=\frac{m}{m-1}u_\varepsilon^{m-1}-c_\varepsilon$ to the regularized equation \er{kseps} leads to
\begin{align}\label{230202}
& \frac{d}{dt}F\left(u_\varepsilon(\cdot,t)\right)+\int_{\R^d} \left
| \frac{2m}{2m-1}\nabla
u_\varepsilon^{m-\frac{1}{2}}-\sqrt{u_\varepsilon}\nabla
c_\varepsilon \right |^2 dx +\frac{4\varepsilon}{m} \int_{\R^d}
\left|\nabla u_\varepsilon^{m/2}\right|^2 dx \nonumber \\
=& \varepsilon c_{d,s}(d-2s) (2s-2) \iint_{\R^d \times \R^d} \frac{u_\varepsilon(x)u_\varepsilon(y)}{\left( |x-y|^2+\varepsilon^2 \right)^{\frac{d+2-2s}{2}}} dxdy
\end{align}
for any $t \in [0,T_w).$ Integrating \er{230202} in time from $0$ to $t$ one obtains
\begin{align}\label{doublestar}
& F\left[u_\varepsilon(\cdot,t)\right]+\int_0^t \int_{\R^d} \left
| \frac{2m}{2m-1}\nabla
u_\varepsilon^{m-\frac{1}{2}}-\sqrt{u_\varepsilon}\nabla
c_\varepsilon \right |^2 dxds \nonumber \\
\le & F\left[u_\varepsilon(0)\right]+\varepsilon c_{d,s}(d-2s) (2s-2) \int_0^t \iint_{\R^d \times \R^d} \frac{u_\varepsilon(x)u_\varepsilon(y)}{\left( |x-y|^2+\varepsilon^2 \right)^{\frac{d+2-2s}{2}}} dxdyds.
\end{align}
From \er{nablaLr} and \er{Linfinity23} we deduce that for any $t \in [0,T_w)$, the following basic estimates hold true:
\begin{align}
&\|u_\varepsilon\|_{L^\infty \left(0,t; L_+^1 \cap L^\infty(\R^d)\right)} \le C, \label{1234} \\
&\left\|\nabla u_\varepsilon^{\frac{m+r-1}{2}} \right\|_{L^2(0,t; L^2(\R^d))} \le C,~~\mbox{for}~~1 \le r<\infty. \label{12345}
\end{align}
Here $C$ represent constants depending only on $\|u_\varepsilon(0)\|_{L^1 \cap L^\infty(\R^d)}$. We shall use \er{1234} and \er{12345} to pass to the limit $\varepsilon \to 0$ in \er{doublestar}.

Firstly, the term in the right hand side of \er{doublestar} is bounded with the help of \er{1234} and the HLS inequality \er{fh} that
\begin{align}\label{230207}
\int_0^t \iint_{\R^d \times \R^d} \frac{u_\varepsilon(x)u_\varepsilon(y)}{\left( |x-y|^2+\varepsilon^2 \right)^{\frac{d+2-2s}{2}}} dxdyds \le C \int_0^t \|u_\varepsilon \|_{L^{2d/(d-2+2s)}(\R^d)}^2 ds \le C
\end{align}
for any $t \in [0,T_w).$

Secondly, the dissipation term is uniformly bounded
\begin{align}
\int_0^{t} \int_{\R^d} \left
| \frac{2m}{2m-1}\nabla
u_\varepsilon^{m-\frac{1}{2}}-\sqrt{u_\varepsilon}\nabla
c_\varepsilon \right |^2 dx ds \le C.
\end{align}
Actually, taking $r=m$ in \er{12345} yields $\left\|\nabla u_\varepsilon^{m-1/2} \right\|_{L^2(0,t;L^2(\R^d))} \le C$. Besides, by Lemma \ref{uceps} and the bound \er{1234} we find
\begin{align}
\int_0^t \left\| u_\varepsilon |\nabla c_\varepsilon|^2 \right\|_{L^1(\R^d)} ds \le C \int_0^t \|u_\varepsilon\|_{L^{\frac{3d}{d-2(1-2s)}}(\R^d)}^3 ds \le C.
\end{align}
Now by a restatement of the proof for \cite[Theorem 2.11]{BL13}, it is straightforward from the convergence properties \er{conver1} to pass to the limit $\varepsilon \to 0$ (without relabeling) and thus obtain
the lower semi-continuity of the dissipation term
\begin{align}\label{dissipationae}
& \int_{0}^t \int_{\R^d} \left|\frac{2m}{2m-1}\nabla
u^{m-1/2}-\sqrt{u}\nabla c \right|^2 dx ds \nonumber \\
\le & \displaystyle
\liminf_{\varepsilon \to 0} \int_{0}^t \int_{\R^d}
\left|\frac{2m}{2m-1}\nabla
u_\varepsilon^{m-1/2}-\sqrt{u_\varepsilon}\nabla c_\varepsilon
\right|^2 dxds.
\end{align}

Finally, the convergence of the free energy can be directly derived from \er{conver1} that
\begin{align}\label{Fstrong}
& F(u_\varepsilon(\cdot,t))= \frac{1}{m-1} \int_{\R^d} u_\varepsilon^m dx-\frac{c_{d,s}}{2} \iint_{\R^d \times \R^d} \frac{u_\varepsilon(x,t)u_\varepsilon(y,t)}{|x-y|^{d-2s}} dxdy \to \nonumber \\
& \frac{1}{m-1} \int_{\R^d} u^m dx-\frac{c_{d,s}}{2} \iint_{\R^d \times \R^d} \frac{u(x,t)u(y,t)}{|x-y|^{d-2s}} dxdy=F(u(\cdot,t))~~\mbox{a.e.~in}~(0,T_w).
\end{align}
Hence the evidences \er{230207}, \er{dissipationae} and \er{Fstrong} allow us to let $\varepsilon \to 0$ in \er{doublestar} such that
\begin{align}
&F \left[ u(\cdot,t) \right]+ \int_{0}^t \int_{\R^d}\left |
\frac{2m}{2m-1}\nabla u^{m-\frac{1}{2}}-\sqrt{u}\nabla c \right
|^2dxds \nonumber\le F[u_0(x)],  \quad a.e. ~~t \in (0,T_w).
\end{align}
Thus $u$ is a free energy solution and the free energy inequality \er{Fuequality} holds. $\Box$
\end{proof}

\section{A variant to the Hardy-Littlewood-Sobolev inequality} \label{MHLS}

In this section, we are going to show the existence of the extremal function for the modified HLS inequality \er{VHLS} which will be used to prove dynamical behaviors of solutions to \er{fracks}. Here several steps are required to the proof which borrows some arguments from \cite{BJ09}. Firstly, we establish a variant to the HLS inequality.
\begin{lemma}\label{Cstar}
Let $\frac{2d}{d+2s}<m<2-\frac{2s}{d}.$ For $u \in L_+^1\cap L^m(\R^d),$ we put
\begin{align}\label{hu}
h(u):=\iint_{\R^d \times \R^d} \frac{u(x)u(y)}{|x-y|^{d-2s}} dxdy.
\end{align}
Then
\begin{align}\label{max}
C_*:=\sup_{u \neq 0}  \frac{h(u)}{\|u\|_{L^1(\R^d)}^{\frac{(d+2s)m-2d}{d(m-1)}} \|u\|_{L^m(\R^d)}^{\frac{m(d-2s)}{d(m-1)}}} < \infty.
\end{align}
\end{lemma}
\begin{proof}
Applying the HLS inequality \er{fh} with $f=h=u$ and $\beta=d-2s$ we obtain
\begin{align}\label{CMP}
h(u) = \iint_{\R^d \times \R^d}
\frac{u(x)u(y)}{|x-y|^{d-2s}} dxdy  \le C_{dsm}
\|u\|_{\frac{2d}{d+2s}}^2 \le C_{dsm} \|u\|_{L^1(\R^d)}^{\frac{(d+2s)m-2d}{d(m-1)}} \|u\|_{L^m(\R^d)}^{\frac{m(d-2s)}{d(m-1)}},
\end{align}
where we have used H\"{o}lder's inequality with
$1<\frac{2d}{d+2s}<m$. Consequently, $C_* \le
C_{dsm}$. $\Box$
\end{proof}

We next turn to the existence of maximisers for the VHLS inequality which can be proved by similar arguments as in \cite[Lemma 3.3]{BJ09}.

\begin{lemma}\label{existencecstar}(The existence of $C_*$) \quad Let $\frac{2d}{d+2s}<m<2-\frac{2s}{d}$. There exists a non-negative, radially symmetric and non-increasing function $W \in L^1 \cap L^m(\R^d)$ such that
$
h(W)=C_*
$
with the normalization $\|W\|_{L^1(\R^d)}= \|W\|_{L^m(\R^d)}=1.$ In addition, there is $R>0$ such that $W$ solves the following Euler-Lagrange equation
\begin{align}
\frac{m(d-2s)}{d(m-1)} C_* W^{m-1}=\left\{
                                     \begin{array}{ll}
                                       2 \int_{\R^d} \frac{W(y)}{|x-y|^{d-2s}} dy-\frac{(d+2s)m-2d}{d(m-1)} C_* supp W,~~& |x|<R, \\
0, & |x|>R.
\end{array}
\right.
\end{align}
Moreover, $R<\infty$ for $m>\frac{2d}{d+2s}.$
\end{lemma}
\begin{proof}
For simplicity, we will use $\|u\|_r=\|u\|_{L^r(\R^d)}$ throughout this proof. Setting
\begin{align*}
J(u):=\frac{h(u)}{\|u\|_{1}^{\frac{(d+2s)m-2d}{d(m-1)}} \|u\|_{m}^{\frac{m(d-2s)}{d(m-1)}}},~~u \in L^1 \cap
L^m(\R^d),
\end{align*}
we consider a maximizing sequence $\{u_j\} \in L^1 \cap L^m(\R^d)$ such that
\begin{align*}
\lim_{j \to \infty} J (u_j)=C_*.
\end{align*}
Firstly, we prove that the maximizing sequence $\{u_j\}$ can be assumed to be non-negative, radially symmetric and non-increasing such that $\|u_j\|_1=\|u_j\|_m=1$. The
second step will show that the supremum can be achieved, i.e.
\begin{align*}
\lim_{j \to \infty} u_j=W, \quad J(W)=C_*
\end{align*}
with $\|W\|_1=\|W\|_m=1.$ Finally, the Euler-Lagrange equation solved by $W$ is obtained.

{\it\textbf{Step 1}} (Radially symmetric, non-increasing
assumption) \quad We denote the scaling
\begin{align*}
\bar{u}_j:=\alpha | u_j(\lambda x) |
\end{align*}
with
\begin{align*}
    \left\{
      \begin{array}{ll}
         \lambda=\|u_j\|_1^{\frac{m}{d(m-1)}} \|u_j\|_m^{-\frac{m}{d(m-1)}}  \\
         \alpha=\lambda^{d} \|u_j\|_1^{-1}.
      \end{array}
    \right.
\end{align*}
Then a direct computation shows that
\begin{align*}
\|\bar{u}_j\|_1=1,~~\|\bar{u}_j\|_m=1,
\end{align*}
and
\begin{align*}
    J(\bar{u}_j) & =\frac{ \iint_{\R^d \times \R^d} \frac{\bar{u}_j(x)\bar{u}_j(y)}{|x-y|^{d-2s}}
dxdy}{\|\bar{u}_j\|_{1}^{\frac{(d+2s)m-2d}{d(m-1)}} \|\bar{u}_j\|_{m}^{\frac{m(d-2s)}{d(m-1)}}} =
\frac{\alpha^2}{\lambda^{d+2s}} \iint_{\R^d \times \R^d}
\frac{|u_j(x)| |u_j(y)|}{|x-y|^{d-2s}} dxdy \\
 & = \frac{h(|u_j|)}{\|u_j\|_{1}^{\frac{(d+2s)m-2d}{d(m-1)}} \|u_j\|_{m}^{\frac{m(d-2s)}{d(m-1)}} } = J(|u_j|).
\end{align*}
We see that $J(u_j)$ is invariant under the scaling. Now we claim
that the maximizing sequence ${u_j}$ could be a non-negative, radially
symmetric and non-increasing function. Denoting $u_j^*$ as the symmetric
decreasing rearrangement of $\bar{u}_j$, the Riesz's rearrangement
inequality \cite[pp. 87]{lieb202} implies that
$\|u_j^*\|_1=\|\bar{u}_j\|_1=1,~\|u_j^*\|_m=\|\bar{u}_j\|_m=1$ and
\begin{align*}
J(u_j^*) & =h(u_j^*) \ge h(\bar{u}_j) =J(\bar{u}_j)=J(|u_j|)\ge J(u_j).
\end{align*}
This entails that $\{u_j^*\}$ is also a maximizing sequence and thus we
may assume that the maximizing sequence $u_j$ is a non-negative,
radially symmetric and decreasing function with
$\|u_j\|_1=\|u_j\|_m=1$.

{\it\textbf{Step 2}} (Existence of the supremum) \quad In
this step, we will show that the maximizing sequence is convergent.
In fact, since $u_j$ is non-negative and non-increasing, we find that
\begin{align*}
1=\| u_j \|_1 =\int_{\R^d} u_j(x) dx & = d \alpha_d
\int_{0}^{\infty}
u_j(r) r^{d-1} dr \\
& \ge d \alpha_d \int_0^R u_j(r) r^{d-1} dr \\
& \ge \alpha_d u_j(R) R^d
\end{align*}
which provides $u_j(R) \le \alpha_d^{-1} R^{-d}$. Likewise,
$u_j(R) \le \alpha_d^{-1/m} R^{-d/m} $. So that
\begin{align}
u_j(R) \le G(R) :=C_0 \inf \left\{ R^{-d}, ~R^{-d/m} \right\},
~~\mbox{ for any} ~R>0.
\end{align}
Now we have known that $u_j$ is bounded in $(R,\infty)$, then the
Helly's selection principle \cite[pp. 89]{lieb202} deduces that
there are a sub-sequence of ${u_j}$ (without relabeling for
convenience) such that
\begin{align*}
u_j \to W ~~\mbox{pointwisely},
\end{align*}
where $W$ is a non-negative and non-increasing function. Besides, the
fact $1<\frac{2d}{d+2s}<m$ implies that
\begin{align*}
\left\|G(|x|)\right\|_{2d/(d+2s)}^{2d/(d+2s)}=d\alpha_d
\int_{0}^{\infty} G(r)^{\frac{2d}{d+2s}} r^{d-1} dr =C \left[
\int_0^1 \frac{1}{r^{\frac{d}{m} \frac{2d}{d+2s}}} r^{d-1} dr +
\int_1^\infty \frac{1}{r^{d \frac{2d}{d+2s}}} r^{d-1} dr
\right] < \infty.
\end{align*}
So we can derive from the HLS inequality that $h(G)<\infty$. Then
\begin{align*}
\left\{
  \begin{array}{ll}
    u_j \le G(|x|), \\
    h(G)< \infty,   \\
    u_j \to W~~\mbox{pointwisely}
  \end{array}
\right.
\end{align*}
together with the dominated convergence theorem ensures
\begin{align}
\lim_{j \to \infty} h(u_j)= h(W).
\end{align}
Hence we have $h(W)=C_*$. Finally, we show $J(W)=C_*.$ Since $u_j \to W$ pointwisely, Fatou's lemma tells us that
\begin{align*}
\left\{
  \begin{array}{ll}
    \|W\|_1 \le \displaystyle \liminf_{j \to \infty} \|u_j\|_1 \le 1, \\
    \|W\|_m \le \displaystyle \liminf_{j \to \infty} \|u_j\|_m \le 1
  \end{array}
\right.
\end{align*}
which results in
\begin{align}
C_*=\lim_{j \to \infty} J(u_j) \ge
J(W)=\frac{h(W)}{ \|W\|_{1}^{\frac{(d+2s)m-2d}{d(m-1)}} \|W\|_{m}^{\frac{m(d-2s)}{d(m-1)}} } \ge h(W)=C_*.
\end{align}
Therefore, we conclude that $J(W)=C_*$ which in turn implies that  $\|W\|_1=\|W\|_m=1$.

{\it\textbf{Step 3}} (The Euler-Lagrange equation solved by $W$) \quad Without loss of generality, we may assume that for any $\rho>0$
\begin{align*}
\Omega:=\left\{ x\in \R^d \Big | W(x)>0,~|x|< \rho  \right\}, \\
\partial \Omega:=\left\{ x\in \R^d \Big | W(x)=0,~|x|= \rho  \right\}.
\end{align*}
Consider $\varphi \in C_0^\infty(\Omega),~\lambda>0$ and there exists
$\lambda_0>0$ such that $W+\lambda \varphi $ in $\Omega$ for
$0<\lambda<\lambda_0$. Then a tedious calculation shows that
\begin{align}\label{functional}
0=&\frac{d}{d\varepsilon} \Bigg|_{\varepsilon=0} J(W+\varepsilon \varphi) \nonumber \\
=& \frac{1}{\|W\|_1^{2a} \|W\|_m^{2b}} \Big( 2 \|W \|_1^a \|W \|_m^b \iint_{\R^d \times \R^d} \frac{\varphi(x) W(y)}{|x-y|^{d-2s}} dxdy \nonumber \\
& -a h(W) \|W\|_1^{a-1} \|W\|_m^b \int_{\Omega} \varphi dx - b h(W) \|W\|_m^{b-m} \|W\|_1^{a} \int_{\R^d} W^{m-1} \varphi dx \Big),
\end{align}
where $a=\frac{(d+2s)m-2d}{d(m-1)}, b=\frac{m(d-2s)}{d(m-1)}.$ Using the assumption that $\|W\|_1=\|W\|_m=1,$ \er{functional} implies that $W$ solves the following equation
\begin{align}\label{WW1}
  2 \int_{\R^d} \frac{W(y)}{|x-y|^{d-2s}} dy=\frac{m(d-2s)}{d(m-1)} C_* W^{m-1}+\frac{(d+2s)m-2d}{d(m-1)} C_* supp W,~~ a.e.~~ x \in \Omega.
\end{align}
Since $W \in L^1 \cap L^m(\R^d),$ it follows from the weak Young inequality that $\frac{1}{|x|^{d-2s}} \ast W \in L^r(\R^d)$ for each $r \in \Big( \frac{d}{d-2s},\frac{1}{\frac{1}{m}-\frac{2s}{d}} \Big].$ In particular, $\frac{1}{|x|^{d-2s}} \ast W$ and $W^{m-1}$ both belong to $L^{\frac{m}{m-1}}(\R^d)$ due to $m>\frac{2d}{d+2s}.$ This property together with \er{WW1} allows us to assert $\rho<\infty$, and thus we claim that $\Omega$ is bounded. If not, we choose a sequence of point $|x_n| \to \infty $ with $x_n \in \Omega,$ owing to $m>\frac{2d}{d+2s}$ we have that $\frac{m(d-2s)}{d(m-1)} C_* W(x_n)^{m-1}+\frac{(d+2s)m-2d}{d(m-1)} C_* supp W>0$ whereas the sequence $ \int_{\R^d} \frac{W(y)}{|x_n-y|^{d-2s}} dy \to 0$, a contradiction. $\Box$
\end{proof}

\section{Proof of Theorem \ref{main}} \label{Proofmain}

In this part, with the aid of the characterization of the extremal function $W$ and the free energy, we will proceed to show that $\|u(\cdot,t)\|_{L^m(\R^d)}$ can be bounded from below or above separately provided by small or large initial data in $L^m$ norm, then we establish the global existence and finite time blow-up of solutions to \er{fracks} in Section \ref{sub41} and Section \ref{sub42}.
\begin{proposition}\label{belowabove}
Let $u(x,t)$ be the free energy solution to \er{fracks} and set $a=\frac{(d+2s)m-2d}{2d-2s-dm}$. Under assumption \er{12} and
\begin{align}\label{Fu0}
\|u_0\|_{L^1(\R^d)}^a F(u_0)<\|W\|_{L^1(\R^d)}^a F(W)
\end{align}
on the initial data.
\begin{enumerate}
  \item[\textbf{(i)}] If
  \begin{align}\label{xiaoyum1}
  \|u_0\|_{L^1 (\R^d)}^a \|u_0\|_{L^m (\R^d)}^m< \|W\|_{L^1 (\R^d)}^a \|W\|_{L^m (\R^d)}^m,
  \end{align}
  then there exists a constant $\mu_1<1$ such that the corresponding free energy solution $u$ satisfies that for any $t>0,$
     \begin{align}\label{xiaoyum}
      \|u(\cdot,t)\|_{L^1(\R^d)}^a \|u(\cdot,t)\|_{L^m(\R^d)}^m< \mu_1 \|W\|_{L^1 (\R^d)}^a \|W\|_{L^m (\R^d)}^m.
    \end{align}
  \item[\textbf{(ii)}] If
  \begin{align}\label{dayum1}
  \|u_0\|_{L^1 (\R^d)}^a \|u_0\|_{L^m (\R^d)}^m > \|W\|_{L^1 (\R^d)}^a \|W\|_{L^m (\R^d)}^m,
  \end{align}
  then there exists a constant $\mu_2>1$ such that the corresponding free energy solution $u$ satisfies that for any $t>0,$
\begin{align}\label{dayum}
     \|u(\cdot,t)\|_{L^1(\R^d)}^a \|u(\cdot,t)\|_{L^m(\R^d)}^m > \mu_2 \|W\|_{L^1 (\R^d)}^a \|W\|_{L^m (\R^d)}^m.
\end{align}
\end{enumerate}
\end{proposition}
\begin{proof}
For simplicity, we will use $\|u\|_r$ to replace $\|u\|_{L^r(\R^d)}$. Firstly we propose the variational functional by the expression of $F(u)$ that
\begin{align}\label{Hu}
H(u):=\|u\|_1^a F(u)=\frac{1}{m-1}\|u\|_1^a \|u\|_m^m- \frac{c_{d,s}}{2} \|u\|_1^a \iint_{\R^d \times \R^d} \frac{u(x)u(y)}{|x-y|^{d-2s}} dxdy.
\end{align}
On the one hand, we can infer from \er{max} that \begin{align}\label{06081}
H(u) & \ge \frac{1}{m-1}\|u\|_1^a \|u\|_m^m-\frac{c_{d,s}}{2} C_{\ast} \|u\|_1^{a+a_0} \|u\|_m^{b_0} \nonumber \\
& = \frac{1}{m-1}\|u\|_1^a \|u\|_m^m-\frac{c_{d,s}}{2} C_{\ast} \left( \|u\|_1^{a} \|u\|_m^{m} \right)^\beta,
\end{align}
where the parameters are endowed with
\begin{align}
a_0=\frac{(d+2s)m-2d}{d(m-1)},~~b_0=\frac{m(d-2s)}{d(m-1)},~~\beta=\frac{d-2s}{d(m-1)}.
\end{align}
On the other hand, plugging the radial function $W$ achieved in Lemma \ref{existencecstar} into $H(u),$ $H(W)$ reads
\begin{align}\label{06082}
H(W)& =\frac{1}{m-1}\|W\|_1^a \|W \|_m^m-\frac{c_{d,s}}{2} C_{\ast} \|W\|_1^{a+a_0} \|W \|_m^{b_0} \nonumber \\
& = \frac{1}{m-1}\|W\|_1^a \|u\|_m^m-\frac{c_{d,s}}{2} C_{\ast} \left( \|W\|_1^{a} \|W\|_m^{m} \right)^\beta.
\end{align}
Let us set
\begin{align}
g(x)=\frac{1}{m-1}x-\frac{c_{d,s}}{2} C_{\ast} x^\beta.
\end{align}
Combining \er{06081} with \er{06082} gives
\begin{align}
g \left( \|W\|_1^a \|W \|_m^m \right) =H(W) \ge g \left( \|u\|_1^a \|u \|_m^m \right)
\end{align}
for all $u \in L^1 \cap L^m(\R^d)$. Hence we deduce that the maximum point of $g(x)$ is attained by $x_*=\|W\|_1^a \|W \|_m^m$ where $x_*$ fulfills
\begin{align}
g'(x_*)=\frac{1}{m-1}-\frac{c_{d,s} C_{\ast}}{2} \beta x_*^{\beta-1}=0.
\end{align}
In other words,
\begin{align}\label{xstar}
\|W\|_1^a \|W \|_m^m =\left( \frac{2}{(m-1)c_{d,s} C_{\ast} \beta }  \right)^{\frac{1}{\beta-1}}.
\end{align}
Coming back to $H(u)$, in case that
\begin{align}
\|u_0 \|_1^a F\left(u_0 \right)<\|W\|_1^a F\left(W \right),
\end{align}
there is a $0<\delta<1$ such that
\begin{align}
\|u_0 \|_1^a F\left(u_0 \right)<\delta \|W\|_1^a F\left(W \right).
\end{align}
Because $F(u)$ is non-increasing with respect to time, it turns out to be
\begin{align}\label{0608starstar}
g\left( \|u\|_1^a \|u \|_m^m \right)\le H(u) = \|u\|_1^a F(u) \le \|u_0\|_1^a F(u_0) < \delta \|W\|_1^a F\left(W \right)
\end{align}
after using the mass conservation $\|u_0\|_1=\|u(\cdot,t)\|_1$ for any $t>0.$ Since for any $x<\|W\|_1^a \|W \|_m^m,$ $g(x)$ is a strictly increasing function. So whenever we have
\begin{align}
\|u_0\|_{1}^a \|u_0\|_{m}^m< \|W\|_{1}^a \|W\|_{m}^m,
\end{align}
\er{0608starstar} implies that there exists a $\mu_1<1$ such that for any $t>0,$
\begin{align}
\|u \|_{1}^a \|u\|_{m}^m< \mu_1 \|W\|_{1}^a \|W\|_{m}^m.
\end{align}
Reversely, for any $x>\|W\|_1^a \|W \|_m^m,$ $g(x)$ is strictly decreasing. So if
\begin{align}
\|u_0\|_{1}^a \|u_0\|_{m}^m> \|W\|_{1}^a \|W\|_{m}^m,
\end{align}
we use once more \er{0608starstar} to demonstrate that there exists a $\mu_2>1$ such that
\begin{align}
\|u \|_{1}^a \|u\|_{m}^m> \mu_2 \|W\|_{1}^a \|W\|_{m}^m.
\end{align}
This completes the proof. $\Box$
\end{proof}

\subsection{Proof of global existence $\textbf{\emph{(i)}}$ of Theorem \ref{main}} \label{sub41}

A direct consequence of \er{xiaoyum} and the blow-up criterion is the following global existence result.

\begin{proposition}(Global existence)\label{Lin}
Under assumption \er{12}, \er{Fu0} and \er{xiaoyum1}, the weak solution obtained in Theorem \ref{ueps} exists on $[0,\infty).$
\end{proposition}
\begin{proof}
Firstly, the usefulness of \er{xiaoyum} will be illustrated to show that the $L^r \left(r>m-1+(2s-2)m/d \right)$ norm is uniformly bounded for any $t>0.$ Then the solution can be proved to be uniformly bounded in time by the iterative method. Finally, we demonstrate that the local weak solution achieved in Theorem \ref{ueps} exists globally.

\textbf{(i)} ~~ The theory is based on the following a priori estimate for solutions to \er{fracks}. It is obtained after multiplying the equation by $r u^{r-1}$.
\begin{align}\label{230530}
& \frac{d}{dt} \int_{\R^d} u^r dx+\frac{4mr(r-1)}{(m+r-1)^2} \int_{\R^d} \left| \nabla u^{\frac{m+r-1}{2}} \right|^2 dx \nonumber \\
= &c_{d,s}(d-2s) (2s-2)(r-1) \iint_{\R^d \times \R^d} \frac{u^r(x) u(y)}{ |x-y|^{d+2-2s}} dxdy.
\end{align}
Due to $2s>2,$ we utilize the HLS inequality \er{fh} with $f=u^r(x)$ and $h=u(y)$ such that
\begin{align}\label{230601}
\iint_{\R^d \times \R^d} \frac{u^r(x) u(y)}{|x-y|^{d+2-2s}} dxdy
\le  C_{dsr} \|u^r\|_{L^{q}(\R^d)} \|u\|_{L^{q_1}(\R^d)}
=  C_{dsr} \left\|u \right\|_{L^{\frac{r+1}{1+(2s-2)/d}}(\R^d)}^{r+1},
\end{align}
where $\frac{1}{q}+\frac{1}{q_1}+\frac{d+2-2s}{d}=2$ with $q_1=rq=\frac{r+1}{1+(2s-2)/d}$, $C_{dsr}$ is a bounded constant depending on $d,s,r$. Now choosing $\frac{r+1}{1+(2s-2)/d}>m$ and employing the GNS inequality \cite[(2.4)]{BL13} lead us to obtain
\begin{align}
\left\|u \right\|_{L^{\frac{r+1}{1+(2s-2)/d}}(\R^d)}^{r+1}
\le C \left\|\nabla u^{\frac{m+r-1}{2}} \right\|_{L^2(\R^d)}^{a \theta} \|u \|_{L^{m}(\R^d)}^{a(1-\theta)\frac{m+r-1}{2}},
\end{align}
where
\begin{align}
a=\frac{2(r+1)}{m+r-1},~~\theta=\frac{\frac{m+r-1}{2m}-\frac{(m+r-1)(1+(2s-2)/d)}{2(r+1)}}{\frac{m+r-1}{2m}-\frac{d-2}{2d}  }.
\end{align}
Next we can use Young's inequality with
\begin{align}
a \theta =2\left( 1+\frac{\frac{2-m}{2m}-\frac{2s}{2d}}{\frac{m+r-1}{2m}-\frac{d-2}{2d}} \right)<2,~~\mbox{for}~~m>\frac{2d}{d+2s}
\end{align}
to arrive at
\begin{align}\label{230624}
\|u\|_{L^{\frac{r+1}{1+(2s-2)/d}}(\R^d)}^{r+1} \le \frac{2mr}{c_{d,s}(d-2s) C_{dsr} (2s-2) (m+r-1)^2} \left\|\nabla u^{\frac{m+r-1}{2}} \right\|_{L^2(\R^d)}^2+C \|u\|_{L^m(\R^d)}^{(r+1)(1-\theta)/(1-a\theta/2)},
\end{align}
where the constant $C_{dsr}$ is defined as in \er{230601}.

\textbf{(ii)} ~~To go further and tackle $\|u\|_{L^r(\R^d)}$, using condition \er{xiaoyum} which expresses
\begin{align}
\|u(\cdot,t)\|_{L^m(\R^d)} <C,~~\mbox{for~any}~t>0,
\end{align}
we can infer from \er{230530}, \er{230601} and \er{230624} to state
\begin{align}\label{230602}
\frac{d}{dt}\|u\|_{L^r(\R^d)}^r +\frac{2mr(r-1)}{(m+r-1)^2} \int_{\R^d} \left| \nabla u^{\frac{m+r-1}{2}} \right|^2 dx \le C.
\end{align}
On the other hand, GNS inequality gives rise to
\begin{align}\label{2306022}
\|u\|_{L^r(\R^d)}^r  \le \left\|\nabla u^{\frac{m+r-1}{2}} \right\|_{L^2(\R^d)}^{\frac{2r \theta}{m+r-1}} \|u\|_{L^1(\R^d)}^{r(1-\theta)} \le \left\|\nabla u^{\frac{m+r-1}{2}} \right\|_{L^2(\R^d)}^2 +C\left( \|u\|_{L^1(\R^d)}\right),
\end{align}
where $\theta=\frac{\frac{m+r-1}{2}-\frac{m+r-1}{2r}}{\frac{m+r-1}{2}-\frac{d-2}{2d}}$.
Combining \er{230602} with \er{2306022} allows us to conclude that for any $r>m-1+(2s-2)m/d$
\begin{align}
\|u(\cdot,t)\|_{L^r(\R^d)}<C,~~\mbox{for~any}~t>0.
\end{align}

\textbf{(iii)} ~~We apply Moser iterative method to obtain the uniformly boundedness of the weak solution,
\begin{align}\label{Linftime}
\|u(\cdot,t)\|_{L^\infty(\R^d)}< C \left( \|u_0\|_{L^1(\R^d)},\|u_0\|_{L^\infty(\R^d)} \right),~~\mbox{for any}~~t>0.
\end{align}
Recalling Proposition \ref{energy}, there are $T_w$ and a free energy solution $u(x,t)$ to \er{fracks} in $[0,T_w)$ with initial condition $u_0.$ Thanks to \er{Linftime}, it is sufficient to show that $T_w=\infty$ by Theorem \ref{ueps}. Hence we obtain a global free energy solution to \er{fracks}. $\Box$
\end{proof}

\subsection{Proof of finite time blow-up $\textbf{(\emph{ii})}$ of Theorem \ref{main}} \label{sub42}

In this subsection, we start with large initial data and utilize the standard argument relying on the evolution of the second moment of solutions as originally done in \cite{JL92} to prove the finite time blow-up result in Theorem \ref{main}.

\begin{proposition}(Finite time blow-up)
Let $a=\frac{(d+2s)m-2d}{2d-2s-dm}$ and $u(x,t)$ be a weak free energy solution obtained in Theorem \ref{ueps} on $[0,T_w)$. Assume $\int_{\R^d} |x|^2 u_0(x) dx<\infty, \|u_0\|_{L^1(\R^d)}^a F(u_0)<\|W\|_{L^1(\R^d)}^a F(W)$. If
\begin{align}
 \|u_0\|_{L^1 (\R^d)}^a \|u_0\|_{L^m (\R^d)}^m > \|W\|_{L^1 (\R^d)}^a \|W\|_{L^m (\R^d)}^m,
\end{align}
then the solution satisfies
\begin{align}
\frac{d}{dt} m_2(t)=\frac{d}{dt}\int_{\R^d} |x|^2 u(x,t) dx=\left( 2d-\frac{2(d-2s)}{m-1} \right) \int_{\R^d} u^m dx+2(d-2s)F(u),~0<t<T_w.
\end{align}
Here $T_w<\infty$ and $\|u\|_{L^\infty(\R^d)}$ blows up at finite time in the sense that $\displaystyle \limsup_{t \to T_w}\|u(\cdot,t)\|_{L^\infty(\R^d)}=\infty$.
\end{proposition}
\begin{proof}
Here we show the formal computations, the passing to the limit from the approximated problem \er{kseps} can be done on account of \cite[Theorem 2.11]{BL13} without any further complication. For simplicity, we shall use $\|u\|_r$ instead of $\|u\|_{L^r(\R^d)}$.

The main tool is the evolution of the second moment with time. By integrating by parts in \er{fracks} and symmetrising the second term we compute
\begin{align}
\|u\|_1^a \frac{d}{dt} m_2(t) &= \|u\|_1^a \frac{d}{dt} \int_{\R^d} |x|^2 u(x,t) dx= \|u\|_1^a \left(  2d \int_{\R^d} u^m dx-c_{d,s}(d-2s) \iint_{\R^d \times \R^d} \frac{u(x,t)u(y,t)}{|x-y|^{d-2s}} dxdy \right)\nonumber \\
&=\left( 2d-\frac{2(d-2s)}{m-1} \right) \|u\|_m^m \|u\|_1^a +2(d-2s)F(u) \|u\|_1^a.
\end{align}
We next apply Proposition \ref{belowabove} with some $\mu_2>1$ and the fact $m<2-2s/d$ to obtain
\begin{align}
\|u\|_1^a \frac{d}{dt} m_2(t)& < \left(2d-\frac{2(d-2s)}{m-1} \right) \mu_2 \|W\|_m^m \|W\|_1^a +2(d-2s) F(u_0) \|u_0\|_1^a  \nonumber \\
& < \left(2d-\frac{2(d-2s)}{m-1} \right) \mu_2 \|W\|_m^m \|W\|_1^a+
2(d-2s) F(W) \|W\|_1^a \nonumber \\
& = \left(2d-\frac{2(d-2s)}{m-1} \right) (\mu_2-1) \|W\|_m^m \|W\|_1^a<0. \label{m2t}
\end{align}
Here we have used the non-increasing of $F(u)$ in time and the identity
\begin{align}
2(d-2s) F(W) \|W\|_1^a=-\left(2d-\frac{2(d-2s)}{m-1} \right) \|W\|_m^m \|W\|_1^a
\end{align}
which can be derived by combining \er{06082} and \er{xstar}. So we can infer from \er{m2t} that there exists a $T>0$ such that $\displaystyle \lim_{t \to T} m_2(t)=0$.

In addition, using H\"{o}lder's inequality one has
\begin{align}
\int_{\R^d} u(x) dx =& \int_{|x| \le R} u(x) dx+\int_{|x|>R} u(x) dx \nonumber \\
\le & C R^{d} \|u\|_{L^\infty(\R^d)} +\frac{1}{R^2} m_2(t).
\end{align}
Choosing $R=\left( \frac{C m_2(t)}{\|u\|_{L^\infty(\R^d)}} \right)^{\frac{1}{d+2}}$ gives
\begin{align}
\|u_0\|_{1}=\|u\|_{1} \le c \|u\|_{L^\infty(\R^d)}^{\frac{2}{d+2}} m_2(t)^{\frac{d}{d+2}}
\end{align}
which yields
\begin{align}
\displaystyle \limsup_{t \to T} \|u(\cdot,t)\|_{L^\infty(\R^d)} \ge \displaystyle \lim_{t \to T} \frac{\|u_0\|_{1}^{(d+2)/2}}{c m_2(t)^{d/2}}=\infty.
\end{align}
Thus the proof is completed. $\Box$
\end{proof}

\section{Conclusions}\label{conclu}

This paper concerns equation \er{fracks} for the intermediate exponent $\frac{2d}{d+2s}<m<2-\frac{2s}{d}$. It is shown that there is a threshold value which is characterized by the optimal constant of a variant of HLS inequality $ \iint_{\R^d \times \R^d} \frac{f(x)f(y)}{|x-y|^{d-2s}} dxdy \le C_* \|f\|_{L^1(\R^d)}^{\frac{(d+2s)m-2d}{d(m-1)}} \|f\|_{L^m(\R^d)}^{\frac{m(d-2s)}{d(m-1)}} $ to classify the global existence and finite time blow-up of non-negative solutions, and the equality of the modified HLS inequality is given by the radial steady state $W(x)$ of equation \er{fracks}. Precisely, in the case that the initial free energy $F(u_0)$ is less than $\|W\|_{L^1(\R^d)}^{\frac{(d+2s)m-2d}{2d-2s-dm}} F(W)/\|u_0\|_{L^1(\R^d)}^{\frac{(d+2s)m-2d}{2d-2s-dm}},$ there exists a global weak solution satisfying the free energy inequality when the initial data satisfies $\|u_0\|_{L^1 (\R^d)}^{\frac{(d+2s)m-2d}{2d-2s-dm}} \|u_0\|_{L^m (\R^d)}^m< \|W\|_{L^1 (\R^d)}^{\frac{(d+2s)m-2d}{2d-2s-dm}} \|W\|_{L^m (\R^d)}^m$ while the solution blows up in finite time provided by the condition $\|u_0\|_{L^1 (\R^d)}^{\frac{(d+2s)m-2d}{2d-2s-dm}} \|u_0\|_{L^m (\R^d)}^m> \|W\|_{L^1 (\R^d)}^{\frac{(d+2s)m-2d}{2d-2s-dm}} \|W\|_{L^m (\R^d)}^m$.



\end{document}